\documentclass[12pt,leqno]{amsart}
\usepackage{euscript, amssymb, amsmath, amsthm}
\usepackage{epsfig}
\usepackage{graphicx}
\usepackage{subcaption}
\usepackage{yfonts}
\usepackage{caption}
\usepackage{longtable}
\usepackage{dcolumn}
\usepackage{setspace}
\usepackage[most]{tcolorbox}
\definecolor{webred}{rgb}{0.75,0,0}
\definecolor{webgreen}{rgb}{0,0.75,0}
\definecolor{refkey}{gray}{0.75}
\usepackage[pagebackref=true, colorlinks=true, citecolor=blue]{hyperref}

\numberwithin{equation}{section}

\newtheorem{theo}{Theorem}[section]
\newtheorem{lem}{Lemma}[section]

\newtheorem{Def}[theo]{Definition}
\theoremstyle{remark}
\newtheorem{rem}{Remark}[section]

\newcommand{\y}{Y}
\newcommand{\lb}{-1}
\newcommand{\de}{\delta}

\newcommand{\ep}{\varepsilon}
\def\R{{\mathbb{R}}}

\def\d{\displaystyle}
\def\e{{\varepsilon}}
\def\p{\partial}

\date{}


\tcbset{
    frame code={}
    center title,
    left=0pt,
    right=0pt,
    top=0pt,
    bottom=0pt,
    colback=gray!10,
    colframe=white,
    width=\dimexpr\textwidth\relax,
    enlarge left by=0mm,
    boxsep=5pt,
    arc=0pt,outer arc=0pt,
    }

\begin{document}

\title[Nonexistence result for the generalized Tricomi  equation]{Nonexistence result for the generalized Tricomi  equation with  the  scale-invariant damping, mass term  and time derivative nonlinearity}
\author[M.F. Ben Hassen, M. Hamouda, M. A. Hamza  and H.K. Teka]{Moahmed Fahmi Ben Hassen$^{1}$, Makram Hamouda$^{1}$, Mohamed Ali Hamza$^{1}$ and Hanen Khaled Teka$^{1}$}
\address{$^{1}$  Department of Basic Sciences, Deanship of Preparatory Year and Supporting Studies, Imam Abdulrahman Bin Faisal University, P. O. Box 1982, Dammam, Saudi Arabia.}

\medskip

\email{mfhassen@iau.edu.sa (M.F. Ben Hassen)}
\email{mmhamouda@iau.edu.sa (M. Hamouda)} 
\email{mahamza@iau.edu.sa (M.A. Hamza)}
\email{hkteka@iau.edu.sa (H.K. Teka)}

\pagestyle{plain}


\maketitle

\begin{abstract}
In this article, we consider the  damped wave equation  in the \textit{scale-invariant case} with time-dependent speed of propagation, mass term and time derivative nonlinearity. More precisely, we study the blow-up of the solutions to the following equation:
\begin{displaymath}
\d (E) \hspace{1cm} u_{tt}-t^{2m}\Delta u+\frac{\mu}{t}u_t+\frac{\nu^2}{t^2}u=|u_t|^p,
\quad \mbox{in}\ \R^N\times[1,\infty),
\end{displaymath}
that we associate with small initial data. Assuming some assumptions on the mass and damping coefficients, $\nu$ and $\mu>0$, respectively, we prove that blow-up region and the lifespan bound of the solution of $(E)$ remain the same as  the ones obtained for the case without mass, {\it i.e.} $(E)$ with $\nu=0$ which constitutes  itself a shift of the dimension $N$ by $\frac{\mu}{1+m}$ compared to the problem without damping and mass.   Finally, we think that the new bound for $p$ is a serious candidate to the critical exponent which characterizes   the threshold between the blow-up and the global existence regions.
\end{abstract}


\section{Introduction}
\par\quad

We consider the following semilinear wave equation which is characterized by  a scale-invariant damping term, a mass term and a time-dependent speed of propagation:
\begin{equation}
\label{G-sys}
\left\{
\begin{array}{l}
\d u_{tt}-t^{2m}\Delta u+\frac{\mu}{t}u_t+\frac{\nu^2}{t^2}u=|u_t|^p,
\quad \mbox{in}\ \R^N\times[1,\infty),\\
u(x,1)=\e u_0(x),\ u_t(x,1)=\e u_1(x), \quad  x\in\R^N,
\end{array}
\right.
\end{equation}
where $m$ is a nonnegative constant, $p>1$, $\mu, \nu \ge 0$ and $\e$ is a small positive number. We suppose that $u_0$ and $u_1$ are two non-negative, compactly supported functions   on  $B_{\R^N}(0,R), R>0$.

It is worth mentioning that the case $m=\mu= \nu=0$ leads to the  well-known Glassey conjecture for which the critical power $p_G$  is given by
\begin{equation}\label{Glassey}
p_G=p_G(N):=1+\frac{2}{N-1}.
\end{equation}
We recall here that the Glassey exponent $p_G$ is said to be critical in the sense that,  for $p>p_G$, the global existence can be obtained, and, for $p \le p_G$,  the nonexistence  of a global  solution under the smallness of the initial data is known; see e.g. \cite{Hidano1,Hidano2,John1,Sideris,Tzvetkov,Zhou1}.\\

Now, for $\mu > 0$ and $m=\nu=0$, a first blow-up region was obtained in \cite{LT2}  for $1<p\leq p_G(N+2\mu)$. Later,  a substantial improvement was performed in \cite{Palmieri-Tu-2019} using the integral representation formula, where the new bound is given by $p_G(N+ \mu)$ for $\mu \ge 2$.  Recently,  a new amelioration  \cite{Our2} shows that $p_G(N+ \mu)$ is the new critical exponent for all $\mu >0$. We recall  that this candidate has a better chance to be optimal. \\

Let us mention that, for $m=0$ and $\mu, \nu > 0$, the mass term may have no influence on the blow-up bound for the exponent $p$ when $\de \ge 0$, where $\de$ is  given by 
\begin{equation}\label{delta}
\de=(\mu-1)^2-4\nu^2.
\end{equation}
In fact,  the authors in \cite{Palmieri-Tu-2019} show the aforementioned observation when $\de \ge 1$. The latter result was extended to $\de \ge 0$ in \cite{Our3}. More precisely, we find again here the known upper bound for $p$, given by $p_G(N+\mu)$ for all $\de \ge 0$, which is the same as the case of massless ($\nu=0$). Hence, the results in \cite{Our2} and \cite{Our3}, corresponding to the cases with and without mass, respectively, are the same. This fact can be interpreted by saying that, as long $\nu > 0$ guarantees the 
validity of the condition  $\de \ge 0$, it is possible to consider the mass term as a lower order perturbation
that does not effect the blow-up condition for $p$. Indeed, for other nonlinearities, as for example the power nonlinearity or the combined nonlinearities, the same observation holds true, namely the massless influence in the dynamics of the solution for large time, see e.g. \cite{Our3, Palmieri2-2019, Palmieri-Reissig, Palmieri-Tu-2019}. \\

Now,  for $m>0$ and $\mu=\nu=0$, the  blow-up results and lifespan estimate of the solution  of \eqref{G-sys} are proven in \cite{LP-Tricomi}. More precisely, it is asserted that the blow-up occurs   for all $1<p \le p_G(N(1+m))$. In the same period but independently and with different methods,  an improvement is obtained in \cite{Lai2020} showing that the new  region is bounded by a plausible candidate for the critical exponent, namely  
\begin{equation}\label{pT}
p \le p_T(N,m):=1+\frac{2}{(1+m)(N-1)-m}, \ \text{for} \ N \ge 2.
\end{equation}
Using different approaches and as an application of the results obtained for the case with mixed nonlinearities, similar blow-up results as in \cite{Lai2020} are derived in \cite{Our5} for the problem \eqref{G-sys} with $m>0$ and $\mu=\nu=0$. Similar results for the combined nonlinearities are obtained in \cite{CLP,Our5}.\\

We aim in this work to the study the blow-up of the solutions of   the Cauchy problem (\ref{G-sys}) for $m,\mu, \nu^2>0$. In this case, the emphasis will be on the comprehension of the influence of the  parameter $m$, which characterizes the time-dependent speed of propagation, on the critical exponent and the lifespan estimate in comparison with the case $m=0$ studied in \cite{Our3}. Indeed, we will show that the usual observation, proved in \cite{Our3} for  (\ref{G-sys}) with $m=0$ and stating that the mass term has no influence in the upper bound for the range of $p$, holds true for  (\ref{G-sys}) with $m>0$ as well. Furthermore, the results in the present work for $m \ge 0$ constitute somehow an extension of \cite{HHP1,HHP2} where the negative time-dependent speed of propagation ($-1<m<0$) is investigated when $\nu =0$. We refer the reader to \cite{Palmieri, Palmieri2, Tsutaya} for more details about the case $-1<m<0$. Very recently, the extension to any $m<0$ is performed in \cite{Tsutaya3, Tsutaya2}.\\

The  structure of the paper is as follows. In Section \ref{sec-main}, we introduce   the energy solution of (\ref{G-sys}), and  we state our main theorem.  Then,  in Section \ref{aux}, some useful lemmas are presented; these will be used in the proof of the main result which will be the subject of Section \ref{sec-ut}. Finally, we show in the appendix in Section \ref{appendix2} several numerical properties of the solution to (\ref{G-sys}) using the associated functionals.

\section{Main Results}\label{sec-main}
\par

This section is devoted to the statement of the main result. However, before that we need to define the energy solution associated with (\ref{G-sys}).
\begin{Def}\label{def1}
Let  $u$ be such that
\begin{equation}\label{u-hyp}
u\in \mathcal{C}([1,T),H^1(\R^N))\cap \mathcal{C}^1([1,T),L^2(\R^N)) \ 
\text{and} \ u_t \in L^p_{loc}((1,T)\times \R^N),
\end{equation}
verifies, for all $\Phi\in \mathcal{C}_0^{\infty}(\R^N\times[1,T))$ and all $t\in[1,T)$, the following identity:
\begin{equation}
\label{energysol2}
\begin{array}{l}
\d\int_{\R^N}u_t(x,t)\Phi(x,t)dx-\int_{\R^N}u_t(x,1)u_1(x)dx  -\int_1^t  \int_{\R^N}u_t(x,s)\Phi_t(x,s)dx \,ds\vspace{.2cm}\\
\d+\int_1^t  \int_{\R^N}s^m \nabla u(x,s)\cdot\nabla\Phi(x,s) dx \,ds+\int_1^t  \int_{\R^N}\frac{\mu}{s}u_t(x,s) \Phi(x,s)dx \,ds\vspace{.2cm}\\
\d  +\int_1^t  \int_{\R^N}\frac{\nu^2}{s^2}u(x,s) \Phi(x,s)dx \,ds=\int_1^t \int_{\R^N}|u_t(x,s)|^p\Phi(x,s)dx \,ds,
\end{array}
\end{equation}
and the condition $u(x,1)=\varepsilon u_0(x)$ is fulfilled in $H^1(\mathbb{R}^N)$. Then, $u$ is called a weak solution of 
 (\ref{G-sys}) on $[1,T)$.
 
A straightforward computation shows that \eqref{energysol2} is equivalent to
\begin{equation}
\begin{array}{l}\label{energysol2-1}
\d \int_{\R^N}\big[u_t(x,t)\Phi(x,t)- u(x,t)\Phi_t(x,t)+\frac{\mu}{t}u(x,t) \Phi(x,t)\big] dx \vspace{.2cm}\\
\d +\int_1^t  \int_{\R^N}u(x,s)\left[\Phi_{tt}(x,s)-s^{2m}\Delta \Phi(x,s) -\frac{\partial}{\partial s}\left(\frac{\mu}{s}\Phi(x,s)\right)+\frac{\nu^2}{s^2} \Phi(x,s)\right]dx \,ds\vspace{.2cm}\\
\d =\int_{1}^{t}\int_{\R^N}|u_t(x,s)|^p\psi(x,s)dx \, ds + \e \int_{\R^N}\big[-u_0(x)\Phi_t(x,1)+\left(\mu u_0(x)+u_1(x)\right)\Phi(x,1)\big]dx.
\end{array}
\end{equation}
\end{Def}

Now, we define the multiplier $M(t)$ as follows:
\begin{equation}
\label{test1}
M(t):=t^{\mu}.
\end{equation}
Hence, we  rewrite Definition \ref{def1}, using $M(t)\Phi(x,t)$ as a test function, in  the following manner.
\begin{Def}\label{def2}
 We say that $u$ is a weak  solution of
 (\ref{G-sys}) on $[1,T)$
if $u$ satisfies \eqref{u-hyp} and
\begin{equation}
\label{energysol}
\begin{array}{l}
M(t)\d\int_{\R^N}u_t(x,t)\Phi(x,t)dx-\ep \int_{\R^N}u_1(x)\Phi(x,1)dx \vspace{.2cm}\\
\d -\int_1^t M(s) \int_{\R^N}u_t(x,s)\Phi_t(x,s)dx \,ds-\int_1^t s^{2m}M(s) \int_{\R^N} u(x,s)\Delta\Phi(x,s) dx \,ds\vspace{.2cm}\\
\d+\int_1^t  \int_{\R^N}\frac{\nu^2M(s)}{s^2}u(x,s) \Phi(x,s)dx \,ds=\int_1^t M(s) \int_{\R^N}|u_t(x,s)|^p\Phi(x,s)dx \,ds,
\end{array}
\end{equation}
 for all $\Phi\in \mathcal{C}_0^{\infty}(\R^N\times[1,T))$ and $t\in[1,T)$, and the condition $u(x,1)=\varepsilon u_0(x)$ is satisfied in $H^1(\mathbb{R}^N)$.
\end{Def}

In the following, we will state the main result in this article.

\begin{theo}
\label{blowup}
Let $\nu^2, \mu \ge 0, m\ge 0$  and $\de \ge 0$ (given by \eqref{delta}). Assume that  $u_0\in H^1(\R^N)$ and $u_1\in L^2(\R^N)$ are non-negative and compactly supported functions   on  $B_{\R^N}(0,R)$ which
  do not vanish everywhere and verify
  \begin{equation}\label{hypfg}
  \frac{\mu-1-\sqrt{\de}}{2}u_0(x)+u_1(x) > 0.
  \end{equation}
Then,  there exists $\e_0=\e_0(u_0,u_1,N,R,p,m,\mu,\nu)>0$ such that for any $0<\e\le\e_0$ the solution $u$  to \eqref{G-sys} which satisfies 
$$\mbox{\rm supp}(u)\ \subset\{(x,t)\in\R^N\times[1,\infty): |x|\le R+\phi_m(t)-\phi_m(1)\},$$
blows up in finite time $T_\e$, where
\begin{equation}\label{xi}
\phi_m(t):=\frac{t^{1+m}}{1+m},
\end{equation} 
and for $N \ge 2$
\begin{displaymath}
T_\e \leq
\d \left\{
\begin{array}{ll}
 C \, \e^{-\frac{2(p-1)}{2-((1+m)(N-1)-m+\mu)(p-1)}}
 &
 \ \text{for} \
 1<p<p_{T}(N+\frac{\mu}{m+1},m), \vspace{.1cm}
 \\
 \exp\left(C\e^{-(p-1)}\right)
&
 \ \text{for} \ p=p_{T}(N+\frac{\mu}{m+1},m),
\end{array}
\right.
\end{displaymath}
 where $p_{T}(N,m)$ is given by \eqref{pT}.\\
 Moreover, for $N=1$, we have
\begin{displaymath}
T_\e \leq
\d \left\{
\begin{array}{ll}
 C \, \e^{-\frac{2(p-1)}{2-(\mu-m)(p-1)}}
 & \text{for} \
1< p < p_{\mu,m}:= \left\{
\begin{array}{ll}
1+\frac{2}{\mu-m} &\text{if} \ m<\mu \le m+\frac{2}{p-1}, \\
\infty  &\text{if} \ \mu \le m,
\end{array}
\right.
 \\
 \exp\left(C\e^{-(p-1)}\right)
&
 \ \text{for} \ 1<p=1+\frac{2}{\mu-m}.
\end{array}
\right.
\end{displaymath}
In the above lifespan estimates, the constant $C$ is positive and independent of $\e$.
\end{theo}

\begin{rem}
It is known that the mass term has no influence on the blow-up region. Once again we confirm here this property by showing in  Theorem \ref{blowup} that the upper bound of the lifespan does not depend on the mass parameter $\nu$. Furthermore, we believe that the critical value for $p$ obtained in Theorem \ref{blowup} is indeed optimal. However, a rigorous confirmation should be carried out through global existence results.
\end{rem}

\begin{rem}
It is notable here that the dimensions $N=1,2$ are somehow particular compared to higher dimensions. Indeed, in dimension $1$, we notice  a competition between the damping and Tricomi terms. Moreover, in dimension $2$, the Tricomi coefficient is not involved in the lifespan upper bound; we observe that the  critical exponent $p_{T}(2,m)$ does not depend on $m$ and hence $T_\e$ does not as well. However, for $N \ge 3$, the two aforementioned parameters are  of course implied in the blow-up.
\end{rem}

\section{Auxiliary results}\label{aux}
\par

One of the fundamental tools that we will employ in the present work is the use of an adequate test function. 
Hence, inspired by some previous works \cite{Our3,Our5,HHP1,HHP2}, we introduce the  positive test function  $\psi(x,t)$ solution to the conjugate equation corresponding to the linear problem, namely
\begin{equation}\label{lambda-eq}
\partial^2_t \psi(x, t)-t^{2m}\Delta \psi(x, t) -\frac{\partial}{\partial t}\left(\frac{\mu}{t}\psi(x, t)\right)+\frac{\nu^2}{t^2}\psi(x, t)=0.
\end{equation}
More precisely, on the one hand, the problem \eqref{G-sys} with constant speed of propagation, namely $m=0$, is studied in \cite{Our3}, and on the other hand, the massless problem with Tricomi term ($t^{2m}$), \textit{i.e.} \eqref{G-sys} with $\nu=0$, is investigated in \cite{Our5}. Therefore, in the present work, the study of \eqref{G-sys} constitutes somehow an extension of the results obtained in \cite{Our3} when $m>0$. Hence, it would be interesting to analyze the influence of the Tricomi coefficient $m$ on the critical exponent and the lifespan when the blow-up occurs. \\
Now, we define the test function $\psi(x,t)$ as follows:
\begin{equation}
\label{test11}
\psi(x,t):=\rho(t)\Psi(x);
\quad
\Psi(x):=
\left\{
\begin{array}{ll}
\d\int_{S^{N-1}}e^{x\cdot\omega}d\omega & \mbox{for}\ N\ge2,\vspace{.2cm}\\
e^x+e^{-x} & \mbox{for}\  N=1,
\end{array}
\right.
\end{equation}
where $\Psi(x)$ is introduced in \cite{YZ06} and verifies $\Delta\Psi=\Psi$, and $\rho(t)$   is solution to
\begin{equation}\label{lambda}
\frac{d^2 \rho(t)}{dt^2}-t^{2m}\rho(t)-\frac{d}{dt}\left(\frac{\mu}{t}\rho(t)\right)+\frac{\nu^2}{t^2}\rho(t)=0.
\end{equation}

The construction of a solution of \eqref{lambda} is performed in Appendix \ref{appendix1}. Moreover,  we list in the following lemma some useful properties of the function $\rho(t)$ that will be employed in the proof of our main result. 

\begin{lem}\label{lem-supp}
There exists a solution $\rho(t)$ of \eqref{lambda} which verifies the following properties:
\begin{itemize}
\item[{\bf (i)}] The function $\rho(t)$ is positive  on $[1,\infty)$. Furthermore,   there exists a constant $C_1$ such that   
\begin{equation}\label{est-rho}
C_1^{-1}t^{\frac{\mu-m}{2}} \exp(-\phi_m(t)) \le  \rho(t) \le C_1 t^{\frac{\mu-m}{2}} \exp(-\phi_m(t)), \quad \forall \ t \ge 1, 
\end{equation}
where $\phi_m(t)$ is defined by \eqref{xi}.
\item[{\bf (ii)}] For all $m>0$, we have
\begin{equation}\label{lambda'lambda1}
\d \lim_{t \to +\infty} \left(\frac{\rho'(t)}{t^m \rho(t)}\right)=-1.
\end{equation}
\end{itemize}
\end{lem}
\begin{proof}
See Appendix \ref{appendix1} for the detail of the proof.
\end{proof}


Note that the notation $C$ stands for a generic positive constant depending on the data ($p,m,\mu,\nu,N,u_0,u_1,\ep_0$) but not on $\ep$.  Obviously, the value of $C$ may change while keeping the same notation along this work. When necessary, we will specify  the dependency of the constant $C$.\\

The following lemma  holds true for the function $\Psi(x)$ defined in \eqref{test11}.
\begin{lem}[\cite{YZ06}]
\label{lem1} Let  $r>1$.
There exists a constant $C=C(m,N,R,r)>0$ such that
\begin{equation}
\label{psi}
\int_{|x|\leq R+\phi_m(t)-\phi_m(1)}\Big(\Psi(x)\Big)^{r}dx
\leq Ce^{r\phi_m(t)}(1+\phi_m(t))^{\frac{(2-r)(N-1)}{2}},
\quad\forall \ t\ge 1.
\end{equation}
\end{lem}

\par
In what follows we will introduce two functionals which are devoted to prove the main results. More precisely, let
\begin{equation}
\label{F1def}
G_1(t):=\int_{\R^N}u(x, t)\psi(x, t)dx,
\end{equation}
and
\begin{equation}
\label{F2def}
G_2(t):=\int_{\R^N}\p_tu(x, t)\psi(x, t)dx.
\end{equation}
The above functionals satisfy some uniform lower bounds in the sense that $\ep^{-1}t^{m}G_1(t)$ and $\ep^{-1}G_2(t)$ are coercive. Indeed, similarly as obtained in the case with mass term and constant speed of propagation \cite{Our3}, we will show here that the functional $G_1(t)$ is coercive starting from a given time independent of the initial data size, and that the functional  $G_2(t)$ may have some oscillations around zero to end up with a coercivity after a relatively large time going to $\infty$ as $\ep \to 0$. This is expected since the problem \eqref{G-sys} constitutes somehow an extension of the one considered in \cite{Our3}.


\begin{lem}
\label{F1}
Assume that the hypotheses in Theorem \ref{blowup} are fulfilled for $u$  a solution of the system \eqref{G-sys}. Then, there exists $T_0=T_0(m,\mu,\nu)>2$ such that 
\begin{equation}
\label{F1postive}
G_1(t)\ge C_{G_1}\, \e t^{-m}, 
\quad\text{for all}\ t \ge T_0,
\end{equation}
where $C_{G_1}$ is a positive constant depending probably  on $u_0$, $u_1$, $N,m, \mu$ and $\nu$.
\end{lem}
\begin{proof} 
Let $t \in [1,T)$. Substituting $\Phi(x, t)$ by $\psi(x, t)$  in \eqref{energysol2-1}, then using \eqref{lambda-eq} and \eqref{test11}, we get
\begin{equation}
\begin{array}{l}\label{eq5}
\d \int_{\R^N}\big[u_t(x,t)\psi(x,t)- u(x,t)\psi_t(x,t)+\frac{\mu}{t}u(x,t) \psi(x,t)\big]dx
\vspace{.2cm}\\
\d=\int_1^t\int_{\R^N}|u_t(x,s)|^p\psi(x,s)dx \, ds 
+\d \e \, C(u_0,u_1),
\end{array}
\end{equation}
where 
\begin{equation}\label{Cfg}
C(u_0,u_1):=\int_{\R^N}\big[\big(\mu\rho(1)-\rho'(1)\big)u_0(x)+\rho(1)u_1(x)\big]\Psi(x)dx.
\end{equation}
From \eqref{lmabdaK} and \eqref{lambda'lambda2}, we see that
\begin{equation}\label{Cfg1}
\mu\rho(1)-\rho'(1)=\frac{\mu-1-\sqrt{\de}}{2}K_{\frac{\sqrt{\de}}{2(1+m)}}(\phi_m(1))+K_{\frac{\sqrt{\de}}{2(1+m)} +1}(\phi_m(1)),
\end{equation}
and consequently  we have
\begin{align}\label{Cfg}
C(u_0,u_1)&=K_{\frac{\sqrt{\de}}{2(1+m)}}(\phi_m(1)) \int_{\R^N} \big[\frac{\mu-1-\sqrt{\de}}2u_0(x)+u_1(x)\big]\Psi(x)dx \\&\ +K_{\frac{\sqrt{\de}}{2(1+m)} +1}(\phi_m(1))\int_{\R^N}u_0(x)\Psi(x)dx. \nonumber
\end{align}
Therefore the positivity of  the constant $C(u_0,u_1)$ is straightforward from  \eqref{hypfg}.\\
Recalling the definition of $G_1$, given by \eqref{F1def},  and \eqref{test11}, then  \eqref{eq5} implies that
\begin{equation}
\begin{array}{l}\label{eq6}
\d G_1'(t)+\Gamma(t)G_1(t)=\int_1^t\int_{\R^N}|u_t(x,s)|^p\psi(x,s)dx \, ds +\e \, C(u_0,u_1),
\end{array}
\end{equation}
where 
\begin{equation}\label{gamma}
\Gamma(t):=\frac{\mu}{t}-2\frac{\rho'(t)}{\rho(t)}.
\end{equation}
Taking the product of  \eqref{eq6} by $\d \frac{t^\mu}{\rho^2(t)}$ and integrating on $(1,t)$, we infer that
\begin{align}\label{est-G1}
 G_1(t)
\ge G_1(1)\frac{\rho^2(t)}{t^\mu}+{\e}C(u_0,u_1)\frac{\rho^2(t)}{t^\mu}\int_1^t\frac{s^\mu}{\rho^2(s)}ds.
\end{align}
Using the definition of  $\phi_m(t)$, given by \eqref{xi}, the positivity of $G_1(1)$ and   \eqref{lmabdaK},  the estimate \eqref{est-G1} yields
\begin{align}\label{est-G1-1}
 G_1(t)
\ge {\e}C(u_0,u_1) t K^2_{\frac{\sqrt{\de}}{2(1+m)}}\left(\phi_m(t)\right)\int^t_{t/2}\frac{1}{sK^2_{\frac{\sqrt{\de}}{2(1+m)}}\left(\phi_m(s)\right)}ds, \quad \forall \ t \ge 2.
\end{align}
Thanks to \eqref{Kmu}, we deduce that there exists $T_0=T_0(m,\mu,\nu)>2$ such that 
\begin{align}\label{est-double}
\phi_m(t)K^2_{\frac{\sqrt{\de}}{2(1+m)}}(\phi_m(t))>\frac{\pi}{4} e^{-2\phi_m(t)} \quad \text{and}  \quad \phi_m(t)^{-1}K^{-2}_{\frac{\sqrt{\de}}{2(1+m)}}(\phi_m(t))>\frac{1}{\pi} e^{2\phi_m(t)}, \ \forall \ t \ge T_0/2.
\end{align}
Combining \eqref{est-double} and  \eqref{est-G1-1}, and remembering \eqref{xi},   we see that
\begin{align}\label{est-U-2}
 G_1(t)
&\ge \e \frac{C(u_0,u_1)}{4}t^{-m}e^{-2\phi_m(t)}\int^t_{t/2}\phi_m'(s)e^{2\phi_m(s)}ds.
\end{align}
Hence, we conclude that
\begin{align}\label{est-G1-3}
 G_1(t)
\ge \e \kappa C(u_0,u_1)t^{-m}, \ \forall \ t \ge T_0,
\end{align}
where $\d \kappa=\kappa(m)$ is a positive constant.

The proof of Lemma \ref{F1} is thus completed.
\end{proof}

In the next lemma, we will derive a negative lower bound for the functional $G_2(t)$ which is independent of $\ep$. As aforementioned, the functional $G_2(t)$ will not be positive for all $t \ge 1$; see Appendix \ref{appendix2} for some numerical simulations.

\begin{lem}
\label{F1-2}
Assume the existence of an energy solution $u$ of the system \eqref{G-sys}  with initial data satisfying the hypotheses as in Theorem \ref{blowup}. Then, there exists $\ep_0>0$ such that, for all $t \in (1,T)$, we have 
\begin{align}\label{G2+bis7}
\d G_2(t) +\mathcal{K} (\nu^2+ \nu^{\frac{2p}{p-1}})  e^{\frac{2p-1}{p-1}\phi_m(t)} t^{\frac{(m+1)(N-1)+3m}{2}} \ge  0,
\end{align}
where $\mathcal{K}$ is a positive constant which may depend on $u_0$, $u_1$, m, p, $N,R,  \ep_0$ and $\mu$ but not on $\ep$ and $\nu$.
\end{lem}
\begin{proof} 
Let $t \in [1,T)$. Choosing $\psi_0(x, t):=e^{-\frac{t^{m+1}-1}{m+1}}\Psi(x)$ (which satisfies $\partial_t \psi_{0}(x, t)=-t^m \psi_0(x, t)$), where $\Psi(x)$ is given by \eqref{test11}, as a test function in \eqref{energysol}, and recalling $M(t)=t^{\mu}$ and $\Delta\psi_0=\psi_0$, we obtain
\begin{equation}
\begin{array}{l}\label{eq4}
\d M(t)\int_{\R^N}u_t(x,t)\psi_0(x,t)dx
-\e\int_{\R^N}u_1(x)\psi_0(x,1)dx \vspace{.2cm}\\
\d+\int_1^tM(s)s^m\int_{\R^N}\left\{
u_t(x,s) \psi_0(x,s)-s^{m}u(x,s)\psi_0(x,s)\right\}dx \, ds \vspace{.2cm}\\
\d+\int_1^t  \int_{\R^N}\frac{\nu^2M(s)}{s^2}u(x,s) \psi_0(x,s)dx \,ds=\int_1^tM(s)\int_{\R^N}|u_t(x,s)|^p\psi_0(x,s)dx \, ds.
\end{array}
\end{equation}
Let
\begin{equation}
\label{F1def1}
F_1(t):=\int_{\R^N}u(x, t)\psi_0(x, t)dx, 
\end{equation}
and
\begin{equation}
\label{F2def1}
F_2(t):=\int_{\R^N}u_t(x, t)\psi_0(x, t)dx.
\end{equation}
Observe that 
\begin{equation}\label{eq5-1-f1f2}
\d F_1'(t)=\int_{\R^N}\left\{
u_t(x,t) \psi_0(x,t)-t^{m}u(x,t)\psi_0(x,t)\right\}dx,
\end{equation}
 then  the above equation with \eqref{eq4} yield
\begin{equation}
\begin{array}{l}\label{eq5-1-0}
\d M(t)(F_1'(t)+t^{m}F_1(t))-\e\int_{\R^N}u_1(x)\psi_0(x,1)dx+\int_1^tM(s)s^m F_1'(s) ds
\vspace{.2cm}\\
\d +\int_1^t  \frac{\nu^2M(s)}{s^2}F_1(s) \,ds =\int_1^tM(s)\int_{\R^N}|u_t(x,s)|^p\psi_0(x,s)dx \, ds,
\end{array}
\end{equation}
It is easy to see that $F_1(t)$ satisfies 
\begin{equation}\label{eq5-1-new}
\d \int_1^tM(s)s^mF_1'(s) ds=- \int_1^t (M(s)s^m)'F_1(s) ds+M(t)t^mF_1(t)-F_1(1).
 \end{equation}
 Combining \eqref{eq5-1-0} and \eqref{eq5-1-new}, we obtain
\begin{equation}
\begin{array}{l}\label{eq5-1}
\d M(t)(F_1'(t)+2t^{m}F_1(t))
-{\e}C_0(u_0,u_1) +\int_1^t  \frac{\nu^2M(s)}{s^2}F_1(s) \,ds \vspace{.2cm}\\
\d=\int_1^t(M(s)s^m)'F_1(s) ds+\int_1^tM(s)\int_{\R^N}|u_t(x,s)|^p\psi_0(x,s)dx \, ds,
\end{array}
\end{equation}
where 
$$C_0(u_0,u_1):=\int_{\R^N}\left\{u_0(x)+u_1(x)\right\}\Psi(x)dx.$$
Again, from the definitions of $F_1$ and  $F_2$, given  by \eqref{F1def1} and  \eqref{F2def1}, respectively, and the equality \eqref{eq5-1-f1f2} rewritten as
 \begin{equation}\label{def231-bis}\d F_1'(t) +t^{m}F_1(t)= F_2(t),\end{equation}
 the equation  \eqref{eq5-1} gives
\begin{equation}
\begin{array}{l}\label{eq5bis1}
\d M(t)(F_2(t)+t^{m}F_1(t))
-{\e}C_0(u_0,u_1) +\int_1^t  \frac{\nu^2M(s)}{s^2}F_1(s) \,ds \vspace{.2cm}\\
\d=\int_1^t(M(s)s^m)'F_1(s) ds+\int_1^tM(s)\int_{\R^N}|u_t(x,s)|^p\psi_0(x,s)dx \, ds.
\end{array}
\end{equation}
By differentiating \eqref{eq5bis1} in time and employing \eqref{def231-bis}, we get
\begin{align}\label{F1+bis1}
\frac{d}{dt} \left\{F_2(t)M(t)\right\}+   2M(t)t^{m}F_2(t)
= M(t)t^{m}\left(F_2(t)+t^{m}F_1(t)\right)-\frac{\nu^2M(t)}{t^2}F_1(t)\\+M(t)\int_{\R^N}|u_t(x,t)|^p\psi_0(x,t)dx.\nonumber
\end{align}
Combining   \eqref{eq5bis1} and \eqref{F1+bis1}, we have
\begin{align}\label{F1+bis3}
\begin{array}{l}
\d \frac{d}{dt} \left\{F_2(t)M(t)\right\}+   2M(t)t^{m}F_2(t)
=\d  {\e}C_0(u_0,u_1) \, t^{m} \vspace{.2cm}\\
\d +t^{m}\int_1^tM(s)\int_{\R^N}|u_t(x,s)|^p\psi_0(x,s)dx \, ds+M(t)\int_{\R^N}|u_t(x,t)|^p\psi_0(x,t)dx\vspace{.2cm}\\
\d +\Sigma_1(t)+\nu^2\Sigma_2(t)+\nu^2\Sigma_3(t),
\end{array}
\end{align}
where $(M(t)=t^{\mu})$, and
\begin{equation}\label{sigma11}
\d \Sigma_1(t):=\d   t^{m}\int_1^t (M(s)s^m)'F_1(s) ds=\d (\mu+m)  t^{m}\int_1^t s^{\mu+m-1}F_1(s) ds,
\end{equation}
\begin{equation}\label{sigma11--2}
\d \Sigma_2(t):=\d  - t^{m}\int_1^t  \frac{M(s)}{s^2}F_1(s) \,ds= - t^{m}\int_1^t  s^{\mu-2}F_1(s) \,ds,
\end{equation}
and
\begin{equation}\label{sigma21}
\Sigma_3(t):=-\frac{M(t)}{t^2}F_1(t)=-  t^{\mu-2}F_1(t).
\end{equation}
Using the fact that $G_1(t)=e^{\phi_m(t)-\phi_m(1)} \rho(t)F_1(t)$ together with the positivity of $G_1(t)$, in view of \eqref{est-G1}, we easily obtain that   $\Sigma_1(t) \ge 0$.\\
By integrating  \eqref{def231-bis}, we have
\begin{equation}\label{def2313}\d F_1(t)= F_1(1)e^{\phi_m(1)-\phi_m(t)} +e^{-\phi_m(t)} \int_1^t e^{\phi_m(s)}F_2(s)ds.
\end{equation}
Inserting the above identity in   \eqref{sigma11--2} and  integrating by parts, we infer that
\begin{align}
\begin{array}{c}
\d \int_1^t  s^{\mu-2}F_1(s) \,ds= F_1(1)\int_1^t s^{\mu-2}e^{\phi_m(1)-\phi_m(s)} ds\vspace{.2cm}\\
 \d + \left(\int_1^t s^{\mu-2}e^{-\phi_m(s)}ds\right) \left(\int_1^t e^{\phi_m(s)}F_2(s)  ds\right)- \int_1^t e^{\phi_m(s)}F_2(s)\left(\int_1^s \tau^{\mu-2}e^{-\phi_m(\tau)} d\tau\right) ds.
\end{array}
\end{align}
Thanks to the boundedness of $\d \int_1^t s^{\mu-2}e^{-\phi_m(s)}ds$  independently of $t$ since it can be easily estimated by above by $\d \int_1^{\infty} s^{\mu-2}e^{-\phi_m(s)}ds \le C$, we obtain
\begin{equation}\label{sigma111}
|\Sigma_2(t)| \le C F_1(1)t^{m} + Ct^{m} \int_1^t e^{\phi_m(s)}|F_2(s)| ds.
\end{equation}
Now, for the term $\Sigma_3(t)$, performing similar estimates as we did for $\Sigma_2(t)$ and using  \eqref{def2313}, we end up with an analogous estimate as \eqref{sigma111}, namely
\begin{equation}\label{sigma1112}
|\Sigma_3(t)| \le C F_1(1)t^{m} + Ct^{m} \int_1^t e^{\phi_m(s)}|F_2(s)| ds ;
\end{equation}
where we recall here that $\d F_1(1)=\ep \int_{\R^N}u_0(x)\Psi(x)dx$ thanks to the definition \eqref{F1def1}.\\
Plugging \eqref{sigma111} and \eqref{sigma1112} in \eqref{F1+bis3} and using the fact that $M(t) \ge 1$, we get
\begin{align}\label{F1+bis45}
\begin{array}{c}
\d \frac{d}{dt} \left\{F_2(t)M(t)\right\}+   2M(t)t^{m}F_2(t)
\ge \d  t^{m}\int_1^t\int_{\R^N}|u_t(x,s)|^p\psi_0(x,s)dx \, ds\vspace{.2cm}\\
\d -\d \tilde{C} \ep_0\nu^2t^{m}  - \tilde{C} \nu^2 t^{m} \int_1^t e^{\phi_m(s)}|F_2(s)| ds,
\end{array}
\end{align}
where $\tilde{C}=\tilde{C}(u_0,u_1,\mu, m, N)$.\\
Thanks to \eqref{F2def1}, the definition of $F_2(t)$, and Lemma \ref{lem1}, we see that
\begin{equation}\label{f2-less}
\begin{array}{rcl}
\d \tilde{C}\nu^2 e^{\phi_m(t)}|F_2(t)| &\le&\d  \int_{\R^N}|u_t(x,t)|^p\psi_0(x,t)dx + C\nu^{\frac{2p}{p-1}}e^{\frac{p}{p-1} \phi_m(t)}\int_{|x|\leq R+\phi_m(t)-\phi_m(1)}\psi_0(x,t)dx \vspace{.2cm}\\
&\le& \d \int_{\R^N}|u_t(x,t)|^p\psi_0(x,t)dx + C \nu^{\frac{2p}{p-1}} e^{\frac{2p-1}{p-1}\phi_m(t)}(\phi_m(t))^{\frac{N-1}{2}}.
\end{array}
\end{equation}
By simply integrating in time \eqref{f2-less}, \eqref{F1+bis45} becomes
\begin{align}\label{F1+bis5}
\begin{array}{rcl}
\d \frac{d}{dt} \left\{F_2(t)M(t)\right\}+   2M(t)t^mF_2(t)
+C \nu^2t^{m}+ C \nu^{\frac{2p}{p-1}} e^{\frac{2p-1}{p-1}\phi_m(t)} t^{\frac{(m+1)(N-1)}{2}+m}  \ge 0,
\end{array}
\end{align}
that we rewrite as follows:
\begin{align}\label{F1+bis4}
\frac{d}{dt} \left\{e^{2 \phi_m(t)}F_2(t)M(t)\right\}+Ce^{2 \phi_m(t)}\left( \nu^2t^{m}+  \nu^{\frac{2p}{p-1}} e^{\frac{2p-1}{p-1}\phi_m(t)} t^{\frac{(m+1)(N-1)}{2}+m}\right)\ge 0.
\end{align}
Now, integrating  \eqref{F1+bis4} yields
\begin{align}\label{F1+bis5}
F_2(t) +C\frac{e^{-2\phi_m(t)}}{M(t)}\int_1^t e^{2\phi_m(s)}\left( \nu^2s^{m}+  \nu^{\frac{2p}{p-1}} e^{\frac{2p-1}{p-1}\phi_m(s)} s^{\frac{(m+1)(N-1)}{2}+m}\right) ds\ge 0.
\end{align}
Hence, we can easily see that
\begin{align}\label{F1+bis6}
F_2(t) +C \nu^2t^{m-\mu}+ C \nu^{\frac{2p}{p-1}}e^{\frac{2p-1}{p-1}\phi_m(t)} t^{\frac{(m+1)(N-1)}{2}+m-\mu} \ge 0.
\end{align}
Recalling the relationship  $G_2(t)=e^{\phi_m(t)-\phi_m(1)} \rho(t)F_2(t)$, we infer that
\begin{align}\label{G2+bis61}
G_2(t) +C \nu^2t^{m-\mu}e^{\phi_m(t)} \rho(t)+ C \nu^{\frac{2p}{p-1}} \rho(t)e^{\frac{3p-2}{p-1}\phi_m(t)} t^{\frac{(m+1)(N-1)}{2}+m-\mu} \ge 0.
\end{align}
Finally, combining  \eqref{est-rho} and \eqref{G2+bis61}, we obtain \eqref{G2+bis7}.

This concludes the proof of Lemma \ref{F1-2}.
\end{proof}

The purpose of the next lemma is to prove that the functional $G_2(t)$ is coercive for (at least) large times. 

\begin{lem}\label{F11}
Assume that the hypotheses in Theorem \ref{blowup} are fulfilled for $u$  a solution of the Cauchy problem  \eqref{G-sys}. Then, there exists $T_1>1$ such that 
\begin{equation}
\label{F2postive}
G_2(t)\ge C_{G_2}\, \e, 
\quad\text{for all}\ t  \ge  T_1=A\left(1-\ln(\e)\right)^{\frac{1}{1+m}},
\end{equation}
where $C_{G_2}$ and $A$ two positive constants depending probably  on $u_0$, $u_1$, $N,m, \mu$ and $\nu$.
\end{lem}
 
\begin{proof}
Let $t \in [1,T)$. Recall the definitions  \eqref{test11}, \eqref{F1def} and \eqref{F2def}, together with
 \begin{equation}\label{def23}\d G_1'(t) -\frac{\rho'(t)}{\rho(t)}G_1(t)= G_2(t),\end{equation}
 then the equation  \eqref{eq6} yields
\begin{equation}
\begin{array}{l}\label{eq5bis}
\d G_2(t)+\left(\frac{\mu}{t}-\frac{\rho'(t)}{\rho(t)}\right)G_1(t)
=\d \int_1^t\int_{\R^N}|u_t(x,s)|^p\psi(x,s)dx \, ds +\e \, C(u_0,u_1).
\end{array}
\end{equation}
Differentiating in time the above identity and employing \eqref{lambda} and \eqref{def23}, we get
\begin{align}\label{F1+bis2}
\d G_2'(t)+\left(\frac{\mu}{t}-\frac{\rho'(t)}{\rho(t)}\right)G_2(t)+\left(-t^{2m}+\frac{\nu^2}{t^2}\right)G_1(t) =\int_{\R^N}|u_t(x,s)|^p\psi(x,t)dx. \nonumber
\end{align}
Remember the definition of $\Gamma(t)$, given by \eqref{gamma}, we deduce that 
\begin{equation}\label{G2+bis3}
\begin{array}{c}
\d G_2'(t)+\frac{3\Gamma(t)}{4}G_2(t)\ge\Sigma_4(t)+\Sigma_5(t)+\int_{\R^N}|u_t(x,t)|^p\psi(x,t)dx,
\end{array}
\end{equation}
where 
\begin{equation}\label{sigma1-exp}
\Sigma_4(t):=\d \left(-\frac{\rho'(t)}{2\rho(t)}-\frac{\mu}{4t}\right)\left(G_2(t)+\left(\frac{\mu}{t}-\frac{\rho'(t)}{\rho(t)}\right)G_1(t)\right),
\end{equation}
and
\begin{equation}\label{sigma2-exp}
\Sigma_5(t):=\d \left(t^{2m}+\left(\frac{\rho'(t)}{2\rho(t)}+\frac{\mu}{4t}\right) \left(\frac{\mu}{t}-\frac{\rho'(t)}{\rho(t)}\right) -\frac{\nu^2}{t^2}\right)  G_1(t).
\end{equation}
In view of the asymptotic result \eqref{lambda'lambda1} and using \eqref{eq5bis}, we obtain the existence of $\tilde{T}_1=\tilde{T}_1(m,\mu, \nu) \ge T_0$ such that
\begin{equation}\label{sigma1}
\d \Sigma_4(t) \ge C \, \e t^m + \frac{t^m}4 \int_1^t\int_{\R^N}|u_t(x,t)|^p\psi(x,s)dx \, ds, \quad \forall \ t \ge \tilde{T}_1. 
\end{equation}
Similarly, thanks to Lemma \ref{F1} and \eqref{lambda'lambda1}, we deduce the existence of $\tilde{T}_2=\tilde{T}_2(m,\mu, \nu) \ge \tilde{T}_1(m,\mu, \nu)$ ensuring that
\begin{equation}\label{sigma2}
\d \Sigma_5(t) \ge 0, \quad \forall \ t  \ge  \tilde{T}_2. 
\end{equation}
Now, combining \eqref{G2+bis3}, \eqref{sigma1} and \eqref{sigma2}, we deduce that
\begin{equation}\label{G2+bis4}
\begin{array}{l}
\d G_2'(t)+\frac{3\Gamma(t)}{4}G_2(t)\ge C_2 \, \e t^m+\int_{\R^N}|u_t(x,t)|^p\psi(x,t)dx \vspace{.2cm}\\
\d + \frac{t^m}4 \int_1^t\int_{\R^N}|u_t(x,t)|^p\psi(x,s)dx \, ds, \quad \forall \ t  \ge  \tilde{T}_2.
\end{array}
\end{equation}
By ignoring the nonlinear terms in \eqref{G2+bis4}, we simply write 
\begin{equation}\label{G2+bis41}
\begin{array}{l}
\d G_2'(t)+\frac{3\Gamma(t)}{4}G_2(t)\ge C_2 \, \e t^m, \quad \forall \ t  \ge  \tilde{T}_2.
\end{array}
\end{equation}
Integrating \eqref{G2+bis41} over $(\tilde{T}_2,t)$ after multiplication  by $\frac{t^{3\mu/4}}{\rho^{3/2}(t)}$, we obtain
\begin{align}\label{est-G111-bis}
 G_2(t)
\ge G_2(\tilde{T}_2)\frac{\rho^{3/2}(t)}{t^{3\mu/4}}+C_2\,{\e} \frac{\rho^{3/2}(t)}{t^{3\mu/4}}\int_{\tilde{T}_2}^t\frac{s^{m+3\mu/4}}{\rho^{3/2}(s)}ds, \quad \forall \ t  \ge  \tilde{T}_2.
\end{align}
Using \eqref{G2+bis7}, we can bound by below the quantity $G_2(\tilde{T}_2)$ as follows:
\begin{equation}\label{G2sup}
G_2(\tilde{T}_2) \ge - \tilde{\mathcal{K}},
\end{equation}
where $\tilde{\mathcal{K}}:=\mathcal{K} (\nu^2+ \nu^{\frac{2p}{p-1}})  e^{\frac{2p-1}{p-1}\phi_m(\tilde{T}_2)} \tilde{T}_2^{\frac{(m+1)(N-1)+3m}{2}}$.\\
Thanks to \eqref{est-rho}, \eqref{est-G111-bis} and \eqref{G2sup}, we infer the existence of  
 $\tilde{T}=\tilde{T}(m,\mu,\nu):=2\tilde{T}_2$ such that the following estimate holds true
\begin{align}\label{est-G2-12}
 G_2(t)
&\ge \d   -\tilde{\mathcal{K}}_1 t^{-\frac{3m}4} e^{-\frac{3t^{m+1}}{2m+2}}+ C\,{\e}  e^{-\frac{3t^{m+1}}{2m+2}} \int^t_{t/2} s^{m} e^{\frac{3s^{m+1}}{2m+2}}ds,
\end{align}
where $\tilde{\mathcal{K}}_1=\tilde{\mathcal{K}}C_1^{\frac32}$ ($C_1$ is used in \eqref{est-rho}).\\
A simple computation in the integral term in the above inequality gives that
\begin{align}\label{est-G2-123}
 G_2(t)
&\ge \d   -\tilde{\mathcal{K}}_1 t^{-\frac{3m}4} e^{-\frac{3t^{m+1}}{2m+2}}+ C\,{\e} \left(1- e^{-\frac{3t^{m+1}}{2m+2}\left(1- (\frac12)^{m+1}\right)}\right).
\end{align}
Consequently, for $\ep$ small, we deduce that
\begin{align}\label{est-G1-2}
 G_2(t)
\ge  C_{G_2}\,{\e}, \quad \forall \ t \ge T_1:=A\left(1-\ln(\e)\right)^{\frac{1}{1+m}},
\end{align}
 for $A>0$ depending on the parameters but not on $\e$.

 Therefore Lemma \ref{F11} is proven.
\end{proof}


\section{Proof of Theorem \ref{blowup}.}\label{sec-ut}

First, we introduce  the following functional:
\[
H(t):=
\int_{\tilde{T}_3}^t \int_{\R^N}|u_t(x,s)|^p\psi(x,s)dx ds
+C_3 \e,
\]
where $C_3=\min(C_2,8C_{G_2})$; $C_{G_2}$ is given by Lemma \ref{F11}, and $\tilde{T}_3>T_1$ is chosen so that $2t^{m}-\frac{3\Gamma(t)}{4}>0$ for all $t \ge\tilde{T}_3$ (this can be guaranteed by \eqref{gamma} and \eqref{lambda'lambda1}).\\
Now, the functional
$$\mathcal{F}(t):=8G_2(t)-H(t),$$
satisfies
\begin{equation}\label{G2+bis6}
\begin{array}{rcl}
\d \mathcal{F}'(t)+\frac{3\Gamma(t)}{4}\mathcal{F}(t) &\ge& \d \left(2t^{m}-\frac{3\Gamma(t)}{4}\right)\int_{\tilde{T}_3}^t \int_{\R^N}|u_t(x,s)|^p\psi(x,s)dx ds\vspace{.2cm}\\ &+&  \d 7\int_{\R^N}|u_t(x,t)|^p\psi(x,t) dx+C_3 \left(8t^{m}-\frac{3\Gamma(t)}{4}\right) \e\\
&\ge&0, \qquad \forall \ t \ge \tilde{T}_3.
\end{array}
\end{equation}
Integrating \eqref{G2+bis6} on $(\tilde{T}_3,t)$ after multiplication  by $\frac{t^{3 \mu/4}}{\rho^{3/2}(t)}$ yields
\begin{align}\label{est-G111}
 \mathcal{F}(t)
\ge \mathcal{F}(\tilde{T}_3)\frac{\tilde{T}_3^{3 \mu/4}\rho^{3/2}(t)}{t^{3 \mu/4}\rho^{3/2}(\tilde{T}_3)}, \quad \forall \ t \ge \tilde{T}_3,
\end{align}
where $\rho(t)$ is defined by \eqref{lmabdaK}.\\
Consequently, we infer that $\d \mathcal{F}(\tilde{T}_3)=8G_2(\tilde{T}_3)-C_3 \e \ge 8G_2(\tilde{T}_3)-8C_{G_2}\e \ge 0$; the positivity of $\mathcal{F}(\tilde{T}_3)$ is due to Lemma \ref{F11} and the definition of $C_3$, namely $C_3=\min(C_2,8C_{G_2}) \le 8C_{G_2}$. This implies in particular the positivity of $\mathcal{F}(t)$ for all $t \ge \tilde{T}_3$, and hence we have
\begin{equation}
\label{G2-est}
8G_2(t)\geq H(t), \quad \forall \ t \ge \tilde{T}_3.
\end{equation}
Thanks to  H\"{o}lder's inequality, \eqref{psi} and \eqref{F2postive}, we can see that
\begin{equation}
\begin{array}{rcl}
\d \int_{\R^N}|u_t(x,t)|^p\psi(x,t)dx &\geq&\d G_2^p(t)\left(\int_{|x|\leq R+\phi_m(t)-\phi_m(1)}\psi(x,t)dx\right)^{-(p-1)} \vspace{.2cm}\\ &\geq& C G_2^p(t) \rho^{-(p-1)}(t)e^{-(p-1)\phi_m(t)}(\phi_m(t))^{-\frac{(N-1)(p-1)}2}.
\end{array}
\end{equation}
Using \eqref{est-rho}, the above estimate   can be written as follows:
\begin{equation}\label{4.5}
\d \int_{\R^N}|u_t(x,t)|^p\psi(x,t)dx \geq C G_2^p(t) t^{-\frac{\left[(N-1)(m+1)-m+\mu\right](p-1)}{2}}, \ \forall \ t \ge \tilde{T}_3.
\end{equation}
Combining  \eqref{G2-est} and \eqref{4.5}, we conclude that
\begin{equation}
\label{inequalityfornonlinearin}
H'(t)\geq C H^p(t) t^{-\frac{\left[(N-1)(m+1)-m+\mu\right](p-1)}{2}}, \quad \forall \ t \ge \tilde{T}_3.
\end{equation}
Observe that $H(\tilde{T}_3)=C_3 \e>0$, 
then the upper bound of the lifespan estimate (as stated in Theorem \ref{blowup}) can be easily obtained. 

 This ends the proof of Theorem \ref{blowup}.

\appendix
\section{}\label{appendix1}
In this appendix, we will construct and give some properties of   $\rho(t)$ a solution of 
\begin{equation}\label{A-lambda}
\frac{d^2 \rho(t)}{dt^2}-t^{2m}\rho(t)-\frac{d}{dt}\left(\frac{\mu}{t}\rho(t)\right)+\frac{\nu^2}{t^2}\rho(t)=0, \quad t \ge 1.
\end{equation}
Then, we define
\begin{equation}\label{y-rho}
\y(\tau):=\rho(t); \quad \tau=\phi_m(t), \quad (\phi_m(t) \ \text{is defined in} \ \eqref{xi}),
\end{equation}
which satisfies
\begin{equation}\label{y-rho5}
\tau^{2}\frac{d^2 \y(\tau)}{d \tau^2}-\frac{\mu-m}{m+1}\tau\frac{d \y(\tau)}{d \tau}+\left[\frac{\mu+\nu^2}{(1+m)^2}-\tau^{2}\right]\y(\tau)=0, \quad \tau \ge \frac{1}{m+1}.
\end{equation}
Let
\begin{equation}\label{lmabdaK}
\d \y(\tau)= \tau^{\frac{\mu+1}{2m+2}}X(\tau).
\end{equation}
Then, we have
\begin{equation}\label{y-rho6}
\tau^{2}\frac{d^2 X(\tau)}{d \tau^2}+\tau\frac{d X(\tau)}{d \tau}-\left[\frac{\delta^2}{4(1+m)^2}+\tau^{2}\right]X(\tau)=0, \quad \tau \ge \frac{1}{m+1},
\end{equation}
where $\delta$ is given by \eqref{delta}.\\
Hence, the function  $\y(\tau)$ given by
\begin{equation}\label{lmabdaK}
\d \y(\tau)=C_m \tau^{\frac{\tilde{\mu}_m+1}{2}}K_{\frac{\sqrt{\de}}{2(1+m)}}(\tau), \quad \tau \ge \frac{1}{m+1},
\end{equation}
where $C_m=(m+1)^{\frac{\mu+1}{2}}$ and
\begin{equation}\label{bessel}
K_{\xi}(t)=\int_0^\infty\exp(-t\cosh \zeta)\cosh(\xi \zeta)d\zeta,\ \xi\in \mathbb{R},
\end{equation}
 is a solution of \eqref{y-rho5} which represents in fact  the modified Bessel equation of second kind of order $\nu$. \\
Therefore, the expression of $\rho(t)$ is given by
\begin{equation}\label{lmabdaK}
\d \rho(t)=t^{\frac{\mu+1}{2}}K_{\frac{\sqrt{\de}}{2(1+m)}}(\phi_m(t)), \quad \forall \ t \ge 1,
\end{equation}
where $\phi_m(t)$ is defined in \eqref{xi}.\\

Now, we can prove Lemma \ref{lem-supp}.
\begin{proof}[Proof of Lemma \ref{lem-supp}]
First, the positivity of $\rho(t)$ is obvious thanks to \eqref{bessel}. Then, from \cite{Gaunt},  the function $K_{\xi}(t)$ satisfies
\begin{equation}\label{Kmu}
K_{\xi}(t)=\sqrt{\frac{\pi}{2t}}e^{-t} (1+O(t^{-1})), \quad \text{as} \ t \to \infty.
\end{equation}
Combining \eqref{lmabdaK} and \eqref{Kmu}, and again remembering the definition of $\phi_m(t)$, given by \eqref{xi}, and the fact that $m > \lb$, we conclude \eqref{est-rho}. The  assertion {\bf (i)} is thus proven.

Now, to prove {\bf (ii)}, using \eqref{lmabdaK} we observe that 
\begin{equation}\label{lambda'lambda}
\d \frac{\rho'(t)}{\rho(t)}=\frac{\mu+1}{2t}+t^{m}\frac{K'_{\frac{\sqrt{\de}}{2(1+m)}}\left(\phi_m(t)\right)}{K_{\frac{\sqrt{\de}}{2(1+m)}}\left(\phi_m(t)\right)}.
\end{equation}
Exploiting the well-known identity for the modified Bessel function of second kind, 
\begin{equation}\label{Knu-pp}
\frac{d}{dz}K_{\nu}(z)=-K_{\nu+1}(z)+\frac{\nu}{z}K_{\nu}(z),
\end{equation}
and combining \eqref{lambda'lambda} and \eqref{Knu-pp}, we obtain
\begin{equation}\label{lambda'lambda2}
\d \frac{\rho'(t)}{\rho(t)}=\frac{\mu+1+\sqrt{\de}}{2t}-t^{m}\frac{K_{1+\frac{\sqrt{\de}}{2(1+m)}}\left(\phi_m(t)\right)}{K_{\frac{\sqrt{\de}}{2(1+m)}}\left(\phi_m(t)\right)}.
\end{equation}
From \eqref{Kmu} and \eqref{lambda'lambda2}, and using the fact that $m > \lb$, we deduce \eqref{lambda'lambda1}.

This ends the proof of Lemma \ref{lem-supp}.
\end{proof}

\section{}\label{appendix2}
Using Matlab, we will perform in this appendix some simple numerical simulations on the functional $F_2(t)$ using $F_1(t)$. Indeed, by varying the values of the different parameters in the problem (\ref{G-sys}), we will distinguish several cases that exhibit the behavior of the functional $F_2(t)$. The aim here is to confirm the results obtained in Lemmas \ref{F1-2} and \ref{F11}, and  to show the behavior of the functional $F_2(t)$, defined by \eqref{F2def1}, for different values of $\de=(\mu-1)^2-4\nu^2$ (this also gives the dynamics of $G_2(t)$ as well). 

Let us first start by recalling the relationship between $F_1(t)$ and $F_2(t)$  which reads as 
\begin{displaymath}\label{def231}\d F_2(t)=F_1'(t) +t^{m}F_1(t),\end{displaymath}
where $F_1(t)$ verifies \eqref{F1+bis1} but without the nonlinear term. \\
We recall here the identity \eqref{F1+bis1} that we write without the nonlinear term as follows:
\begin{align}\label{F1+bis1--app}
\frac{d}{dt} \left\{F_2(t)M(t)\right\}+   2M(t)t^{m}F_2(t)
= M(t)t^{m}\left(F_2(t)+t^{m}F_1(t)\right)-\frac{\nu^2M(t)}{t^2}F_1(t).
\end{align}
This yields
\begin{align}\label{F1+bis1-app}
\d \frac{\mu}{t}F_2(t) + F'_2(t) + t^{m}F_2(t)
= t^{2m}F_1(t)-\frac{\nu^2}{t^2}F_1(t),
\end{align}
which can be written as
\begin{align}\label{eq-F2}
\d F_1''(t) + \left( \frac{\mu}{t} + 2t^{m}\right)F'_1(t) + \left((m+\mu) t^{m-1}+\frac{\nu^2}{t^2}\right)F_1(t)=0.
\end{align}
We associate with \eqref{eq-F2} two positive initial data $F_1(1), F'_1(1)>0$.\\

The numerical treatment of \eqref{eq-F2} yields the graphs for $F_2(t)$  as shown in Figures 1--6.

\begin{figure}[ht] 
  \label{fig1} 
  \begin{minipage}[b]{0.5\linewidth}
    \centering
    \includegraphics[width=.5\linewidth]{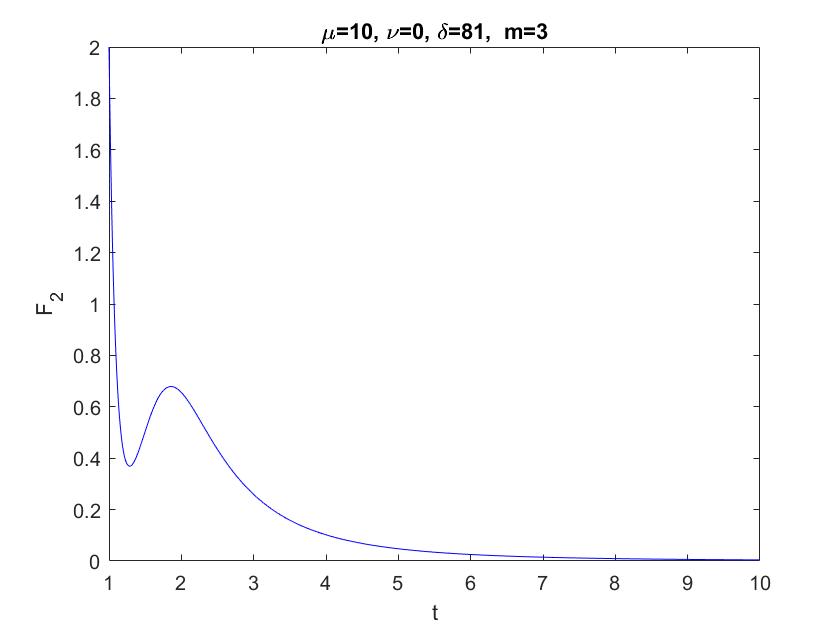} 
    \caption{\tiny{The case $\mu=10, \nu=0, m=3$ (the free-mass case with $\de >0$).}} 
    \vspace{4ex}
  \end{minipage}
  \begin{minipage}[b]{0.5\linewidth}
    \centering
    \includegraphics[width=.5\linewidth]{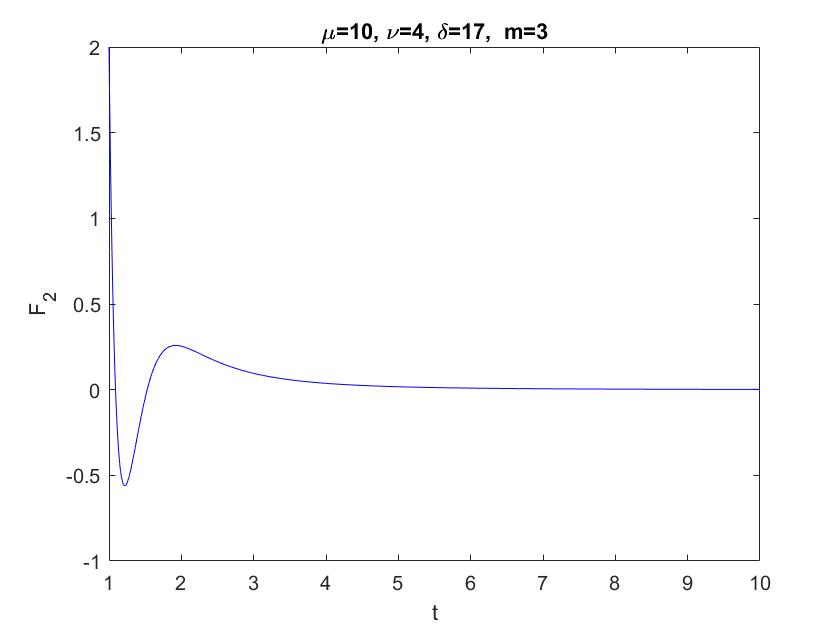} 
    \caption{\tiny{The case $\mu=10, \nu=4, m=3$ which corresponds to $\de >0$.}} 
    \vspace{4ex}
  \end{minipage} 
  \begin{minipage}[b]{0.5\linewidth}
    \centering
    \includegraphics[width=.5\linewidth]{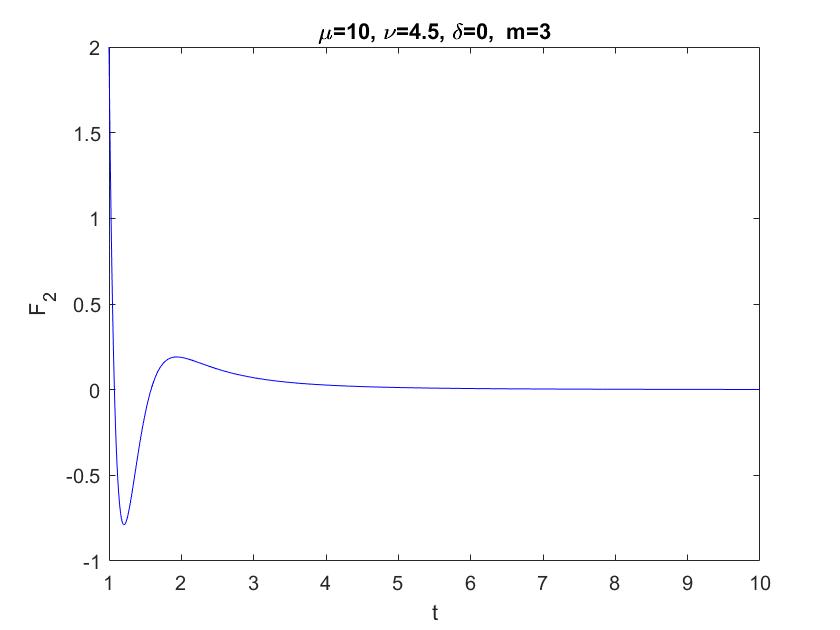} 
    \caption{\tiny{The case $\mu=10, \nu=4.5, m=3$ which corresponds to $\de =0$.}} 
    \vspace{4ex}
  \end{minipage}
  \begin{minipage}[b]{0.5\linewidth}
    \centering
    \includegraphics[width=.5\linewidth]{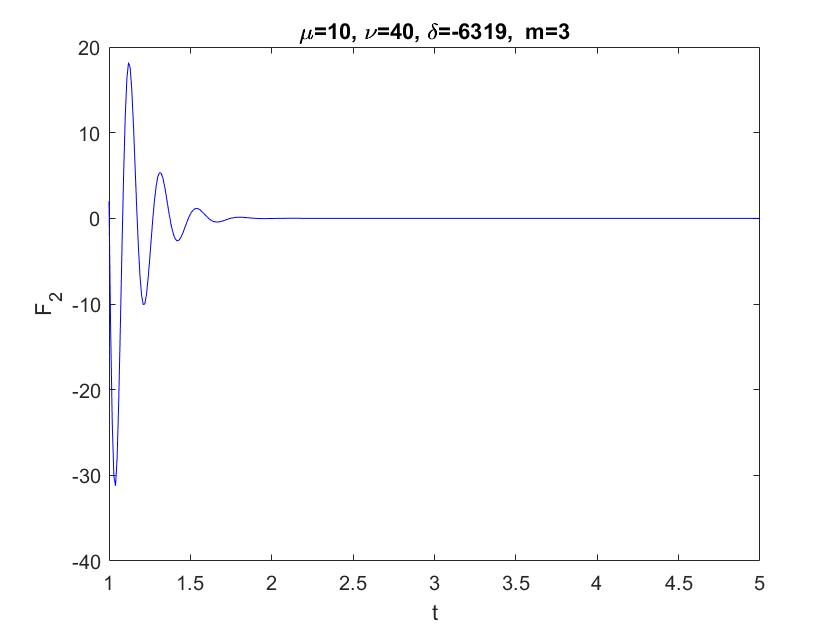} 
    \caption{\tiny{The case $\mu=10, \nu=40, m=3$ which corresponds to $\de <0$.}} 
    \vspace{4ex}
  \end{minipage} 
\end{figure}
\begin{figure}
    \centering
        \minipage{0.32\textwidth}
        \centering
    \includegraphics[width=.5\linewidth]{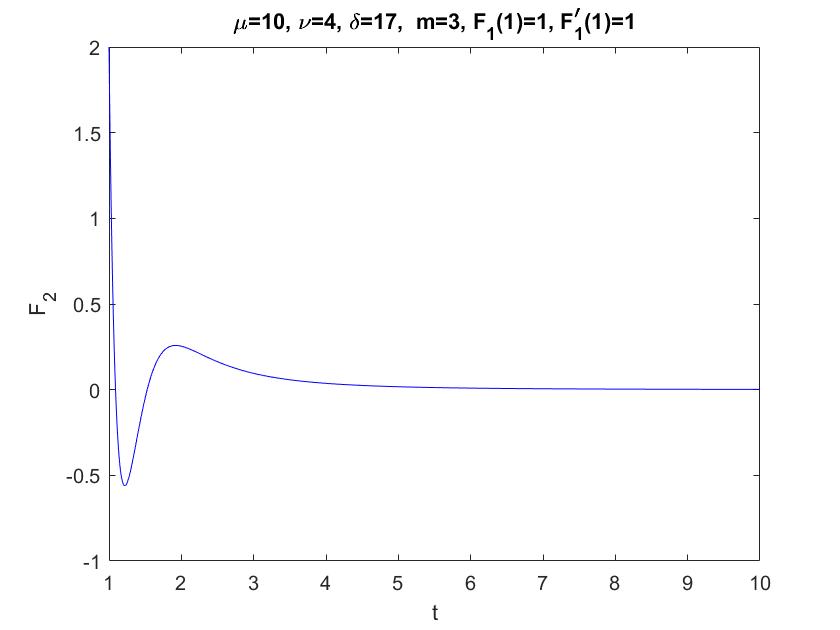} 
    \caption{\tiny{The case $\mu=10, \nu=4, m=3$  which corresponds to $\de >0$ with initial data  $F_1(1)=F'_1(1)=1$.}} 
    \endminipage\hfill
\minipage{0.32\textwidth}
\centering
    \includegraphics[width=.5\linewidth]{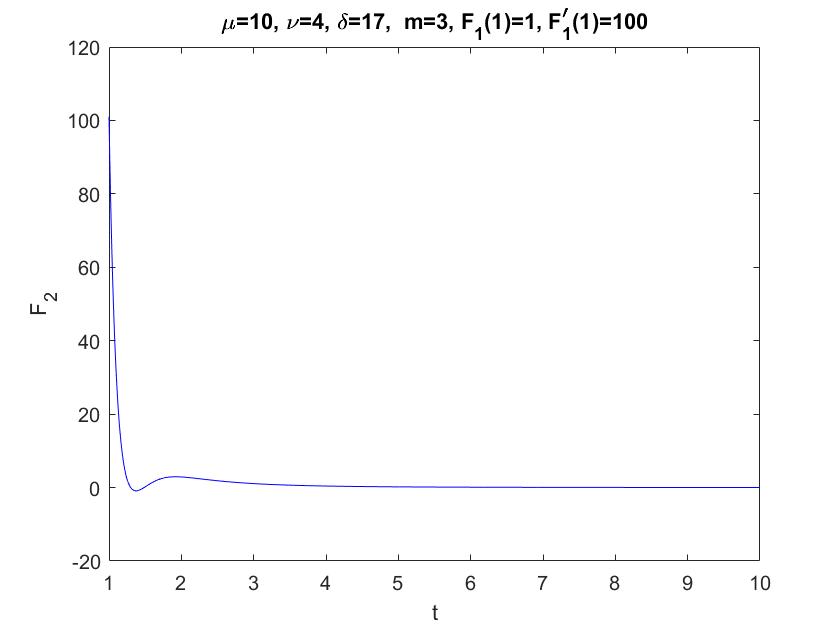} 
    \caption{\tiny{The case $\mu=10, \nu=4, m=3$  which corresponds to $\de >0$ with initial data  $F_1(1)=1, F'_1(1)=100$.}} 
    \endminipage\hfill
\minipage{0.32\textwidth}
\centering
    \includegraphics[width=.5\linewidth]{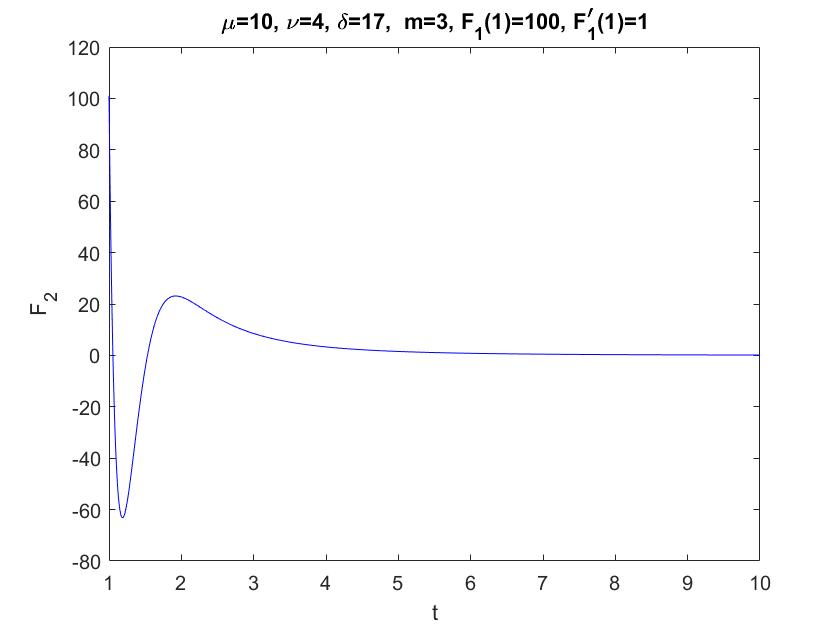} 
    \caption{\tiny{The case $\mu=10, \nu=4, m=3$  which corresponds to $\de >0$ with initial data  $F_1(1)=100, F'_1(1)=1$.}} 
   \endminipage 
\end{figure}

Some observations arise from the graphs of $F_2(t)$:
\begin{itemize}
\item The massless case ($\nu=0$) infers that  $F_2(t)$ is positive for all $t \ge 1$, and consequently the same conclusion holds for $G_2(t)$ (see Figure 1). This observation coincides  with our results in \cite{Our2} where the positivity of $F_2(t)$ and $G_2(t)$ is proved.
\item In Figures 2 and 3, where $\de \ge 0$,  a negative lower bound of $F_2(t)$ is clearly obtained although for relatively large time the functional  $F_2(t)$ is positive. However, some oscillations near $t=1$ are observed. 
\item In Figure  4, corresponding to the case  $\de <0$, we remark more and more oscillations near the initial time $t=1$. Note that the case $\de \le 0$ is out of the target of the present work but may be studied elsewhere.
\item In Figures 5, 6 and 7, we show the influence of the size of the initial data on  the behavior of the functional $F_2(t)$. More precisely, we remark that if the initial data $F_1(1)$ is much larger than $F'_1(1)$ then the negativity of the lower bound of the functional $F_2(t)$ is  enhanced.
\end{itemize}


\bibliographystyle{plain}

\end{document}